\documentclass[12pt,a4paper]{amsart}
\usepackage{amssymb}
\usepackage{amsmath}
\usepackage{xcolor}
\usepackage{amscd}
\usepackage{pstricks}
\usepackage{pst-node}

\usepackage{geometry}
\newtheorem{thm}{Theorem}[section]
\newtheorem{cor}[thm]{Corollary}
\newtheorem{lem}[thm]{Lemma}
\newtheorem{prop}[thm]{Proposition}

\theoremstyle{definition}

\newtheorem{defn}[thm]{Definition}

\usepackage{graphics}
\usepackage{graphicx}

\newtheorem{rem}[thm]{Remark}

\numberwithin{equation}{section}

\begin{document}
%%%%%%%%%%%%%%%%%%%%%%%%%%%%%%%%%%%%%%%%%%%%%%%%%%%%%%%%%%%%%%%%%%%%%%%%

%%%%%%%%%%%%%%%%%%%%%%%%%%%%%%%%%%%%%%%%%%%%%%%%%%%%%%%%%%%%%%%%%%%%%%%%

\title{Heat kernel asymptotics for quaternionic contact manifolds}

\author{ A. Laaroussi}

\thanks{The author has been supported by the priority program SPP 2026 {\it geometry at infinity} of Deutsche Forschungsgemeinschaft (project number BA 3793/6-1).}

\address{Abdellah Laaroussi \endgraf
 Institut f\"{u}r Analysis, Leibniz Universit\"{a}t \endgraf
Welfengarten 1, 30167 Hannover, Germany\endgraf
}
\email{abdellah.laaroussi@math.uni-hannover.de}

%\date{\today}
%%%%%%%%%%%%%%%%%%%%%%%%%%%%%%%%%%%%%%%%%%%%%%%%%%%%%%%%%%%%%%%%%%%%%%%%%
%%%%%%%%%%%%%%%%%%%%%%%%%%%%%%%%%%%%%%%%%%%%%%%%%%%%%%%%%%%%%%%%%%%%%%%%%
\begin{abstract}
In this paper, we study the heat kernel associated to the intrinsic sublaplacian on a quaternionic contact manifold considered as a subriemannian  manifold. More precisely, we explicitly compute the first two coefficients $c_0$ and $c_1$ appearing in the small time asymptotics expansion of the heat kernel on the diagonal. We show that the second coefficient $c_1$ depends linearly on the qc scalar curvature $\kappa$. Finally we apply our results to compact qc-Einstein manifolds and prove the spectral invariance of geometric quantities in the subriemannian setting.
\end{abstract}
%%%%%%%%%%%%%%%%%%%%%%%%%%%%%%%%%%%%%%%%%%%%%%%%%%%%%%%%%%%%%%%%%%%%%%%%%%
%%%%%%%%%%%%%%%%%%%%%%%%%%%%%%%%%%%%%%%%%%%%%%%%%%%%%%%%%%%%%%%%%%%%%%%%%%
\maketitle
\tableofcontents
\thispagestyle{empty}
%%%%%%%%%%%%%%%%%%%%%%%%%%%%%%%%%%%%%%%%%%%%%%%%%%%%%%%%%%%%%%%%%%%%%%%%%%
%%%%%%%%%%%%%%%%%%%%%%%%%%%%%%%%%%%%%%%%%%%%%%%%%%%%%%%%%%%%%%%%%%%%%%%%%%%%%%%%%%%%%%%%%%%%%%
%\todo{{\bf First goal:} 
%\begin{itemize}
%\item 
%In coordinates of $T^* \mathbb{S}^7 \cong T\mathbb{S}^7 \cong\mathbb{S}^7 \times \mathbb{R}^7$ find expressions for: 
%canonical one-form, gradient, symplectic form, Poisson bracket. 
%\item Find first integrals via principal symbols of commuting operators. 
%\item Check whether the methods in \cite{Bar} apply in case of a degenerate metric. 
%\item Study the Lie algebra generated by the symbols of 4,5,6, or 7 vector fields under the Poisson bracket. 
%\item Is the geodesic flow for the different sub-Laplace operators 
%\begin{equation*}
%\Delta^{\textup{sub}}_{j}=- \sum_{\ell =1}^j X_{\ell}^2, \hspace{4ex} j=4,5,6 
%\end{equation*}
%completely integrable (cf. notation below)? Can we find (other) metrics for which the geodesic flow becomes completely integrable.  
%\item Complete integrability of the geodesic flow corresponding to the sub-Laplacian on $\mathbb{S}^7$ induced by the quaternionic Hopf fibration. 
%\item Similar for $\mathbb{S}^3$ and $\mathbb{S}^{15}$. 
%\item Determine the isometry group of the above subriemannian structures. 
%\end{itemize}
% }
 %%%%%%%%%%%%%%%%%%%%%%%%%%%%%%%%%%%%%%%%%%%%%%%%%%%%%%%%%%%%%%%%%%%%%%%%%%%%%%%
 %%%%%%%%%%%%%%%%%%%%%%%%%%%%%%%%%%%%%%%%%%%%%%%%%%%%%%%%%%%%%%%%%%%%%%%%%%%%%%%%%%%%%%%%
 \section{Introduction}
Let $(M,\mathcal{H},\langle\cdot,\cdot\rangle)$ be a subriemannian (SR) manifold, i.e. $M$ is a smooth connected orientable manifold endowed with a bracket generating subbundle $\mathcal{H}$ of the tangent bundle $TM$. Moreover,  $\langle\cdot,\cdot\rangle$ denotes a family of inner products on $\mathcal{H}$ which smoothly vary with the base point. \ \\

Quaternionic contact (qc) manifolds were introduced by Biquard in \cite{Biquard} and are examples of fat SR manifolds \cite{BariIva}. A qc manifold ($M,\mathcal{H},\langle\cdot,\cdot\rangle$) is a $(4n+3)$-dimensional connected manifold endowed with a codimension $3$ distribution $\mathcal{H}$ and a fiberwise inner product $\langle\cdot,\cdot\rangle$ on $\mathcal{H}$. Locally, the distribution $\mathcal{H}$ is given as the kernel of three contact forms $\eta_1,\eta_2,\eta_3$ and there are almost complex structures $I_1,I_2,I_3$ on $\mathcal{H}$ satisfying the quaternionic commutation relations
$$ (I_{i})^2=I_1I_2I_3=-\text{Id for }i=1,2,3.$$
Moreover, the following compatibility conditions hold
$$ 2\langle I_{i}\cdot,\cdot\rangle=d\eta_i(\cdot,\cdot)\text{ on }\mathcal{H}.$$
A classical example of a qc manifold is given by the quaternionic Hopf fibration
$$\mathbb{S}^3\hookrightarrow\mathbb{S}^{4n+3} \rightarrow \mathbb{HP}^n.$$
Here the distribution $\mathcal{H}_q$ at $q\in\mathbb{S}^{4n+3}$ is defined as the orthogonal complement of the vertical  space (tangent space to the fiber through $q$) with respect to the standard Riemannian metric on $\mathbb{S}^{4n+3}$. This standard qc structure on $\mathbb{S}^{4n+3}$ was intensively studied by many authors from the subriemannian geometric point of view \cite{Baud_Wang,BauLaa,MaMo12}. Moreover, this example is included in the sub-class of {\it $3$-Sasakian }manifolds.\ \\

On a qc manifold, there exists a linear connection adapted to the qc structure, the so-called Biquard connection \cite{Biquard} . The Biquard connection plays a similar role as the Levi-Civita connection in Riemannian geometry or the Tanaka-Webster connection on CR manifolds \cite{Tanaka,Webster}.\ \\

Any regular SR structure on $M$ induces a hypoelliptic sublaplacian $\Delta_{sub}$ which intrinsically is defined based on the Popp measure constuction \cite{ABGF,Bariz}. From an analytical point of view on may study the diffusion on $M$ generated by the heat operator induced by $\Delta_{sub}$. A classical problem is to find relations between  analytical invariants (e.g. coefficients appearing in the small time asymptotics of the heat kernel on the diagonal) and geometric invariants of the underlying structure, like in the Riemannian case \cite{ Berger, Gilkey, Rosenberg}.\ \\

In \cite{Bena, Verd} the asymptotic of the heat kernel on the diagonal was proved in the subriemannian case, i.e. the  heat kernel $p(t,\cdot,\cdot)$ induced by the (intrinsic) sublaplacian $\Delta_{sub}$ has the following expansion for small times:
$$p(t,q,q)=\frac{1}{t^{Q/2}}\left(c_0(q)+c_1(q)t+\cdots\right).$$
Here $Q$ denotes the Hausdorff dimension of the metric space $(M,d)$ where $d$ is the subriemannian distance (Carnot Carath\'eodory distance) on $M$.
 As an example, the first coefficient $c_0$ in this expansion can be directly calculated with help of  the heat kernel of the nilpotentization of the structure \cite{Verd}.  However, the explicit geometric meaning of the remaining coefficients  is unknown in the general case. In \cite{Bari}, using the nilpotent approximation, it was shown in the case of a $3$D contact SR manifold, that the second coefficient $c_1$ can be identified with some invariant defined on $3$D contact SR manifolds. Similar results hold for strictly pseudoconvex CR manifolds endowed with a Levi-metric \cite{BGS}. Therein, by developing an appropriate pseudodifferential calculus, the authors showed that the second coefficient $c_1$ can be interpreted as the scalar curvature of the Tanaka-Webster connection.\ \\
 
In this paper, we are interested in the analysis of the intrinsic sublaplacian on qc manifolds. Our goal is to find a relation between the second coefficient $c_1(q)$  and geometric invariants of the underlying SR structure.  We first calculate the Popp measure on such a qc manifold and determine the intrinsic sublaplacian. Moreover, applying recent results due to Y.C. de Verdi\'{e}re, L. Hillairet and E. Tr\'{e}lat in \cite{Verd} combined with the existence of (qc) normal coordinates on qc manifolds \cite{Kunkel}, we show that the second coefficient $c_1$  appearing in the small time asymptotics of the heat kernel associated to the intrinsic sublaplacian depends linearly on the qc scalar curvature. As an application of this result, we prove that the dimension, the Popp volume and the qc scalar curvature of a qc-Einstein compact manifold are spectral invariants.\ \\

The paper is organized as follows:  Section \ref{subgeo} provides basic concepts and definitions in subriemannian geometry. In Section \ref{qc_manifolds} we recall the definition of  qc manifolds and we list some of their properties. Then we compute the Popp volume and the intrinsic sublaplacian induced by a qc structure in Section \ref{Poppmea}. In Section \ref{spectral_invariants} we compute the first coefficients in the small time asymptotics of the heat kernel by using an approximation method  \cite{Chpo,Verd} and the qc normal coordinates \cite{Kunkel}.

 \section{Subriemannian geometry}\label{subgeo}
 We start recalling the basic definitions in subriemannian geometry \cite{ABB,Mo,S1,S2}.\ \\
 
 A subriemannian manifold is a triple $(M,\mathcal{H},\langle \cdot,\cdot\rangle)$ where
 \begin{itemize}
 \item[(a)] $M$ is a connected orientable smooth manifold of dimension $n\geq 3$.
 \item[(b)] $\mathcal{H}$ is a smooth distribution of constant rank $m<n$ which is bracket generating, i.e. if we set for $j\geq 1$
 $$\mathcal{H}^{1} :=\mathcal{H}\text{ and } \mathcal{H}^{j+1}:=\mathcal{H}^{j}+[\mathcal{H},\mathcal{H}^{j}],$$
 then for all $q\in M$ there is $N(q)\in\mathbb{N}$ such that $\mathcal{H}^{N(q)}_{q}=T_qM.$
 \item[(c)] $\langle \cdot,\cdot\rangle$ is a fiber inner product on $\mathcal{H}$, i.e.
 $$\langle \cdot,\cdot\rangle_q:\mathcal{H}_q\times\mathcal{H}_q\longrightarrow\mathbb{R}$$
  is an inner product for all $q\in M$ and it smoothly varies with $q\in M$.
 \end{itemize}

We say that a subriemannian manifold $(M,\mathcal{H},\langle \cdot,\cdot\rangle)$ is regular, if for all $j\geq 1$ the dimension of $\mathcal{H}^{j}_q$ does not depend on the point $q\in M$. Furthermore, $M$ is said to be of step $r$ if $r$ is the smallest integer such that $\mathcal{H}^r=TM$.\ \\
In this work we only consider regular subriemannian manifolds of step $2$. Therefore, in what follows, we recall the required concepts only for this class of manifolds.\ \\
A local  frame $\{X_1,\cdots,X_m,X_{m+1},\cdots,X_n\}$ is called {\it adapted}, if  the vector fields 
$X_1,\cdots,X_m$ form a local orthonormal frame of 
 $(\mathcal{H}, \langle \cdot, \cdot\rangle )$.\ \\

On a subriemannian manifold the definition of a sublaplacian requires the data of a smooth measure. Let $\mu$ be a smooth measure on $M$. We denote by $\text{div}_\mu$ the divergence operator associated with the measure $\mu$, defined by
$$\mathcal{L}_X\mu=\text{ div}_\mu(X)\mu$$
for every smooth vector field $X$ on $M$. Then we can associate to $\mu$ a sublaplacian $\Delta_{sub}^\mu$ defined as the hypoelliptic, second order differential operator
$$\Delta_{sub}^\mu f:=-\text{div}_\mu\left(\nabla_\mathcal{H} f\right)\text{ for }f\in C^\infty(M).$$
Here $\nabla_\mathcal{H}$ denotes the horizontal gradient with respect to the horizontal metric $\langle\cdot,\cdot\rangle$ on $\mathcal{H}$ (see \cite{ABGF,Bariz} for more details).
Since by assumption the subriemannian manifold $M$ is regular, there is a canonical choice of smooth measure on $M$, namely the Popp measure. In this case, the sublaplacian defined from the Popp measure is called the intrinsic sublaplacian.\ \\
 Note that the sublaplacian is positive and if  the manifold $M$ endowed with the subriemannian distance is complete then $\Delta_{sub}^\mu$ is essentially selfadjoint on compactly supported smooth functions and has a unique selfadjoint extension on $L^2(M,\mu)$ (see \cite{Verd}). Therefore the heat semigroup $\left(e^{-t\Delta_{sub}^\mu}\right)_{t>0}$ is a well-defined one-parameter family of bounded operators on $L^2(M,\mu)$. In the following, we denote by $p(t,\cdot,\cdot)$ the heat kernel of the operator $e^{-t\Delta_{sub}^\mu}$ which is smooth due to the hypoellipticity of $\frac{d}{dt}+\Delta_{sub}^\mu$.\ \\
 We recall the following formula for the small time asymptotic expansion for the heat kernel on the diagonal \cite{Bena,Verd}: for all $N\in\mathbb{N}$ and $q\in M$,
 $$p(t,q,q)=\frac{1}{t^{Q/2}}\left(c_0(q)+c_1(q)t+\cdots+c_N(q)t^N+o(t^N)\right)\text{ as }t\to 0. $$
 Moreover, when assuming regularity of the subriemannian manifold, the functions $c_i$ are smooth in a neighborhood of $q$. Here $Q$ denotes the Hausdorff dimension of the metric space $(M,d)$ where $d$ denotes the subriemannian distance (Carnot Carath\'eodory distance) on $M$.
 \section{Quaternionic contact manifolds}\label{qc_manifolds}
 Qc manifolds have been introduced by Biquard in \cite{Biquard}. We recall  the definition of such a manifold.
 \begin{defn}
 A quaternionic contact manifold $(M,\mathcal{H},\langle\cdot,\cdot\rangle)$ is a $(4n+3)$-dimensio-\\
 nal connected manifold $M$ together with a corank $3$ distribution $\mathcal{H}$ and a fiberwise inner product $\langle\cdot,\cdot\rangle$ on $\mathcal{H}$ such that
 \begin{enumerate}
 \item $\mathcal{H}$ is given locally as the kernel of an $\mathbb{R}^3$-valued $1$-form $\eta=(\eta_1,\eta_2,\eta_3)$:
 \begin{equation}\label{ker}
 \mathcal{H}=\bigcap_{i=1}^{3}\text{Ker}(\eta_i).
 \end{equation}
 \item There are three almost complex structures $I_1,I_2$ and $I_3$ on $\mathcal{H}$ that satisfy the quaternionic commutation relations:
 \begin{equation}\label{almost}
 (I_{i})^2=I_1I_2I_3=-\text{Id for }i=1,2,3.
 \end{equation}
 Furthermore, the following compatibility conditions hold:
 \begin{equation}\label{compatibility}
 2\langle I_{i}X,Y\rangle=d\eta_i(X,Y)
 \end{equation}
 for all horizontal vector fields $X,Y\in\mathcal{H}$ and $i=1,2,3$.
 \end{enumerate}
 \end{defn}
 Note that the choice of $1$-forms $\eta_1,\eta_2,\eta_3$ and almost complex structures $I_1,I_2,I_3$ with the above properties (\ref{ker})-(\ref{compatibility}) is not unique. If $\psi\in SO(3)$ then $\psi(\eta)$ and $\psi(I)$ with $I:=(I_1,I_2,I_3)$ as well satisfy the above relations \cite{IMV}. Hence we have a $2$-sphere bundle of almost complex structures over $M$ (locally) given by
 $$\mathbb{I}:=\{aI_1+bI_2+cI_3:a^2+b^2+c^2=1\}.$$ 
 An important fact is that the contact forms $\eta_1,\eta_2,\eta_3$ and the almost complex structures $I_1,I_2,I_3$ uniquely determine the metric $\langle\cdot,\cdot\rangle$ on $\mathcal{H}$. Note also that a qc manifold $(M,\mathcal{H},\langle\cdot,\cdot\rangle)$ is a regular SR manifold of step two \cite{BariIva}.\ \\
 
 The presence of the three almost complex structures and their relation to the metric $\langle\cdot,\cdot\rangle$ on $\mathcal{H}$ provides an $Sp(n)Sp(1)$-structure on the horizontal distribution $\mathcal{H}$. By an $Sp(n)Sp(1)$-frame $\{X_1,\cdots,X_{4n}\}$ we mean an orthonormal frame of the distribution $\mathcal{H}$ such that
 $$I_i X_{4k+1}=X_{4k+i+1}\text{ for }k=0,\cdots,n-1\text{ and }i=1,2,3.$$
 In the following we denote by $X_\mathcal{H}$ (resp. $X_\mathcal{V}$) the orthogonal projection of a vector field $X$ onto $\mathcal{H}$ (resp. $\mathcal{V}$).\ \\
 On a qc manifold with $n\geq 2$ there exists a canonical connection defined by Biquard. In \cite{Biquard} it is shown that there is a unique connection $\nabla$ with torsion $T$ and a unique complementary subspace $\mathcal{V}$ to $\mathcal{H}$ in $TM$ such that
 \begin{enumerate}
 \item $\nabla$ preserves the decomposition $\mathcal{H}\oplus\mathcal{V}$ and the $Sp(n)Sp(1)$-structure on $\mathcal{H}$:
 $$\nabla \langle\cdot,\cdot\rangle=0\text{ and }\nabla\sigma\in\Gamma(\mathbb{I})\text{ for }\sigma\in\Gamma(\mathbb{I}).$$
 \item The torsion $T$ on $\mathcal{H}$ fulfils
 $$T(X,Y)=-[X,Y]_{\mathcal{V}}\text{ for }X,Y\in\mathcal{H}.$$
 \item For $V\in\mathcal{V}$, the endomorphism 
 \begin{align*}
 T(V,\cdot):\mathcal{H}&\longrightarrow\mathcal{H}\\
  X&\longmapsto T(V,X)_\mathcal{H}
 \end{align*}
 lies in $\left(sp(n)\oplus sp(1)\right)^\perp\subset gl(4n)$.
 \item There is a natural identification $\varphi:\mathcal{V}\longrightarrow Sp(1)$ with $\nabla \varphi=0$. 
 \end{enumerate}
 This connection is known as the {\it Biquard connection}. Furthermore, the vertical distribution $\mathcal{V}$ is locally generated by the Reeb  vector fields $V_1,V_2,V_3$ defined by
\begin{equation}\label{reeb}
\eta_i(V_j)=\delta_{ij}\text{ and }(V_i\lrcorner d\eta_j)_\mathcal{H}=-(V_j\lrcorner d\eta_i)_\mathcal{H}\text{ for }i,j=1,2,3.
\end{equation}
Using the Reeb vector fields $V_1,V_2,V_3$ we extend the metric $\langle\cdot,\cdot\rangle$ on $\mathcal{H}$ to a Riemannian metric $g$ by requiring $\mathcal{H}\perp\mathcal{V}$ and 
$$g:=\langle\cdot,\cdot\rangle\oplus\left(\eta_1^2+\eta_2^2+\eta_3^2\right).$$
Note that neither the extended Riemannian metric nor the Biquard connection depend on the action of $SO(3)$ on $\mathcal{V}$.
\begin{rem}
In case of $n=1$, the conditions (\ref{reeb}) are not satisfied in general. Furthermore, it was shown in \cite{Duchemin}, that if we assume in addition the existence of Reeb vector fields as in (\ref{reeb}), then the existence of  a linear connection with similar properties as above is assured. Hence, by a qc manifold with $n=1$, we always mean a qc manifold with conditions (\ref{reeb}).
\end{rem}
Let us denote by $T$ (resp. $R$) the torsion (resp. curvature) tensor of the Biquard connection defined by
\begin{align*}
&T(X,Y):=\nabla_X Y-\nabla_Y X-[X,Y],\\
&R(X,Y)Z:=\nabla_X\nabla_Y Z-\nabla_Y \nabla_X Z-\nabla_{[X,Y]}Z,
\end{align*}
for smooth vector fields $X,Y$ and $Z$. In the following we denote by $\left\{\theta_1,\cdots\theta_{4n}\right\}$ the dual frame of a horizontal frame $\left\{X_1,\cdots,X_{4n}\right\}$.\ \\
It is occasionally convenient to have a notation for the entire frame $\{X_1,\cdots,X_{4n},V_1, \\ V_2,V_3\}$ and coframe $\{\theta_1,\cdots,\theta_{4n},\eta_1,\eta_2,\eta_3\}$. Therefore, we may refer to $V_i$ as $X_{4n+i}$ and $\eta_i$ as $\theta_{4n+i}$. Furthermore, in order to have a consistent index notation we will use different letters for different ranges of indices as follows:
$$a,b,c\in\{1,\cdots,4n+3\},\hspace{2mm}\alpha,\beta,\gamma,\delta\in\{1,\cdots,4n\},\hspace{2mm}i,j,k\in\{1,2,3\},$$
and $\overline{i}:=4n+i$, for $i=1,2,3$.
With this convention we set
$$I_{\alpha\beta}^{i}:=g(I_iX_\alpha,X_\beta),\hspace{2mm} T_{ab}^{c}:=\theta_c\left(T(X_a,X_b)\right)\text{ and }R_{abc}^{d}:=\theta_d\left(R(X_a,X_b)X_c\right).$$
\begin{defn}\label{defi} The quaternionic contact scalar curvature $\kappa$ is defined by
$$\kappa:=\sum_{\alpha,\beta=1}^{4n}R_{\alpha\beta\beta}^{\alpha}.$$
\end{defn}

 We recall the following identities proved in \cite{Biquard,IMV} which we shall need later:
\begin{prop}\label{identities}
It holds:
\begin{enumerate}
\item For $\alpha,\beta=1,\cdots,4n$ and $i=1,2,3$:
$$T_{\alpha\beta}^{\overline{i}}=-2I_{\alpha\beta}^{i};$$
\item $T_{\overline{i}\hspace{0,5mm}\overline{j}}^{\overline{i}}=0$ for $i,j=1,2,3;$
\item For $V\in\mathcal{V}$ and $I\in\mathbb{I}$:
$$\sum_{\alpha=1}^{4n}g(T(V,X_\alpha),IX_\alpha)=0;$$
\item For $i=1,2,3$:
$$\sum_{\alpha,\beta=1}^{4n}g(R(X_\alpha,I_iX_\alpha)I_iX_\beta,X_\beta)=\frac{2n\kappa}{n+2};$$
\item For $i=1,2,3$:
$$\sum_{\alpha,\beta=1}^{4n}g(R(X_\beta,I_iX_\alpha)X_\alpha,I_iX_\beta)=-\frac{n\kappa}{n+2}.$$
\end{enumerate}
\end{prop}

\section{Popp measure and intrinsic sublaplacian}\label{Poppmea}
Let $\{X_1,\cdots,X_{4n},V_1,V_2,V_3\}$ be an orthonormal frame near $q\in M$  such that $\{X_1,\cdots,X_{4n}\}$ is an $Sp(n)Sp(1)$-frame. Then this frame is also an adapted frame for the SR structure $(M,\mathcal{H},\langle\cdot,\cdot\rangle)$. According to \cite{Bariz}, the associated Popp measure $\mathcal{P}$ can be expressed locally in the form 
$$\mathcal{P}=\frac{1}{\sqrt{\det{B}}}\theta_1\wedge\cdots\wedge\theta_{4n}\wedge\eta_1\wedge\eta_2\wedge\eta_3.$$
Here $B=(B_{ij})_{ij}$ is the $3\times 3$-matrix function locally defined near $q$ with coefficients given by
$$B_{ij}:=\sum_{\alpha,\beta=1}^{4n}b_{\alpha\beta}^{i}b_{\alpha\beta}^{j},$$
where $b_{\alpha\beta}^{i}$ are defined for $\alpha,\beta=1,\cdots,4n$ and $i=1,2,3$
 by
 $$b_{\alpha\beta}^{i}:=g([X_\alpha,X_\beta],V_i).$$
 Now by \eqref{compatibility} we can write:
 \begin{align*}
 b_{\alpha\beta}^{i}&=g([X_\alpha,X_\beta],V_i)\\
 &=\eta_i([X_\alpha,X_\beta])\\
 &=-d\eta_i(X_\alpha,X_\beta)\\
 &=-2g(I_iX_\alpha,X_\beta).
 \end{align*}
Hence it follows that
\begin{align*}
B_{ij}&=4\cdot\sum_{\alpha,\beta=1}^{4n}g(I_iX_\alpha,X_\beta)g(I_jX_\alpha,X_\beta)\\
&=4\cdot\sum_{\alpha=1}^{4n}g(I_iX_\alpha,I_jX_\alpha)\\
&=16n\delta_{ij}.
\end{align*}
In the last equality we used the skew-symmetry of the almost complex structures and the commutation relations \eqref{almost}. This shows that $B$ is a diagonal matrix:
$$B=16n\cdot\textup{Id}\in\mathbb{R}^{3\times 3}$$
and hence we obtain the following formula for the Popp measure:
\begin{lem}\label{measure} The Popp measure $\mathcal{P}$ for the quaternionic contact manifold $(M,\mathcal{H},\langle\cdot,\cdot\rangle)$ has the form
$$\mathcal{P}=\frac{1}{(16n)^{3/2}}d\sigma,$$
where $d\sigma:=\theta_1\wedge\cdots\wedge\theta_{4n}\wedge\eta_1\wedge\eta_2\wedge\eta_3$.
\end{lem}
In terms of the orthonormal frame $\{X_1,\cdots,X_{4n},V_1,V_2,V_3\}$, the intrinsic sublaplacian $\Delta_{sub}$ with respect to the Popp measure $\mathcal{P}$ can be expressed in the form (see \cite{Bariz})
\begin{equation}\label{sublaplacian}
\Delta_{sub}=-\left(\sum_{\alpha=1}^{4n}X_\alpha^2+\text{div}_{\mathcal{P}}(X_\alpha)X_\alpha\right).
\end{equation}

Here the divergence operator is defined in terms of the Lie derivative $L_{X_\alpha}$ by
$$L_{X_\alpha}(\mathcal{P})=\text{div}_{\mathcal{P}}(X_\alpha)\mathcal{P}.$$
The computation of the first-order term in the formula (\ref{sublaplacian}) is done like in the $3$D contact case  \cite{Bari}. For this we use the formula
$$L_X(\mu\wedge\nu)=L_X(\mu)\wedge\nu+\mu\wedge L_X(\nu)$$
for any vector field $X$ and differential forms $\mu$ and $\nu$ on $M$. With this in mind, we obtain
$$L_{X_\alpha}(\mathcal{P})=\frac{1}{(16n)^{3/2}}\sum_{a=1}^{4n+3}\theta_1\wedge\cdots\wedge\theta_{a-1}\wedge L_{X_\alpha}(\theta_a)\wedge\theta_{a+1}\wedge\cdots\wedge \theta_{4n+3}.$$
Since $\{\theta_1,\cdots,\theta_{4n+3}\}$ is a coframe, we write for $\alpha=1,\cdots,4n$ and $a=1,\cdots,4n+3$:
$$L_{X_\alpha}(\theta_a)=\sum_{b=1}^{4n+3}f_{\alpha ab}\theta_b.$$
The functions $f_{\alpha ab}$ can be calculated as follows
$$f_{\alpha ab}=L_{X_\alpha}(\theta_a)(X_b)=\theta_a([X_b,X_\alpha])=:c_{b\alpha}^{a}.$$
Hence we obtain
$$\text{div}_{\mathcal{P}}(X_\alpha)=\sum_{\beta=1}^{4n}\theta_\beta([X_\beta,X_\alpha])+\sum_{i=1}^{3}\eta_i([V_i,X_\alpha])=\sum_{a=1}^{4n+3}c_{a\alpha}^{a}.$$
\begin{lem}\label{sublap}
The intrinsic sublaplacian $\Delta_{sub}$ has the expression
$$\Delta_{sub}=-\left(\sum_{\alpha=1}^{4n}X_\alpha^2+\left(\sum_{a=1}^{4n+3}c_{a\alpha}^{a}\right)X_\alpha\right),$$
where $c_{a\alpha}^{a}=\theta_a([X_a,X_\alpha])$.
\end{lem}

\section{First and second heat invariants}\label{spectral_invariants}
In the  setting of general subriemannian manifolds, a powerful method in the analysis of a sublaplacian is given by the so-called nilpotent approximation. The idea consists in an approximation of the subriemannian manifold at a given point by a nilpotent Lie group endowed with a left-invariant subriemannian structure. In the following, we briefly recall the relevant concepts. For more details we refer to \cite{Chpo,Verd}.\ \\

 Let $(M,\mathcal{H},\langle\cdot,\cdot\rangle)$ be a step two regular and complete subriemannian manifold of dimension $n$ and denote by 
$$ \{X_1,\cdots,X_m,X_{m+1},\ldots ,X_n\}$$ 
 a local adapted frame at $q\in M$. A system of local coordinates 
$$\psi: M \supset U_q\longrightarrow \mathbb{R}^n=\mathbb{R}^m\oplus\mathbb{R}^{n-m}$$
 is called {\it linearly adapted at $q$}  if 
$$\psi(q)=0 \hspace{2ex} \text{\it  and }\hspace{2ex} \psi_\ast(\mathcal{H}_q)=\mathbb{R}^m.$$
In a system of linearly adapted coordinates at $q$, we have a notion of nonholonomic orders $"ord"$ corresponding to the natural dilations $\delta_\lambda:\mathbb{R}^n\rightarrow\mathbb{R}^n$ defined  for $\lambda>0$ by
$$\delta_\lambda(x_1,\cdots,x_m,x_{m+1},\cdots,x_n):=(\lambda x_1,\cdots,\lambda x_m,\lambda^2 x_{m+1},\cdots,\lambda^2 x_n).$$
More precisely, we set: 
$$\textup{ord}(x_i):=\begin{cases}
1& \text{if } 1\leq i\leq m,\\
2& \text{if } m+1\leq i\leq n 
\end{cases}$$
and 
$$\textup{ord}\left(\frac{\partial}{\partial x_i}\right):=\begin{cases}
-1& \text{if } 1\leq i\leq m,\\
-2& \text{if } m+1\leq i\leq n.
\end{cases}$$
Using the infinitesimal generator $P$ of the action $\delta_\lambda$ on $\mathbb{R}^n$:
$$P:=\sum_{i=1}^{m}x_i\frac{\partial}{\partial x_i}+2\sum_{i=m+1}^{n}x_i\frac{\partial}{\partial x_i},$$
it is possible to define a homogeneous tensor field $\varphi$ with respect to the above dilations by the condition:
\begin{equation}\label{order}
L_P(\varphi)=l\varphi,
\end{equation}
where $l$ is the order of $\varphi$ and $L_P$ denotes the Lie derivative with respect to $P$. Furthermore, we denote by $\varphi^{(l)}$ the homogeneous part of order $l$ of a tensor field $\varphi$.\ \\
Every smooth vector field $X$ on $\mathbb{R}^n$ has an anisotropic expansion near $0$ of the form (see \cite{Chpo}):
$$X\simeq X^{(-2)}+X^{(-1)}+\cdots,$$
where $X^{(l)}$ is a polynomial vector field of order $l$, i.e. homogeneous of order $l$ with respect to the dilations $\delta_\lambda$.\ \\
In the following we need a special class of linearly adapted coordinates called {\it privileged coordinates}. These are linearly adapted coordinates at $q$ such that every vector field $\psi_\ast(X_i)$ for $1\leq i\leq m$, has an expansion near $0$ where all the homogeneous terms have orders greater or equal than $-1$:
\begin{equation}\label{asymp}
\psi_\ast(X_i)\simeq X_i^{(-1)}+X_i^{(0)}+\cdots.
\end{equation}
An example of privileged coordinates at $q\in M$ is given by the so-called {\it canonical coordinates of the first kind}  defined as the inverse of the local diffeomorphism:
$$(x_1,\cdots,x_n)\longmapsto \exp{(x_1X_1+\cdots+x_n X_n)}(q).$$
Here $X_1,\cdots,X_m,X_{m+1},\cdots,X_n$ is an adapted local frame at $q$. 
\vspace{1mm}\par 
Note that the vector fields $X_1^{(-1)},\cdots,X_m^{(-1)}$ on $\mathbb{R}^n$ generate a graded step two nilpotent Lie algebra $\mathfrak{g}(q)$ called the tangent Lie algebra  at $q$ (see \cite{Chpo}). Let us denote by $(\mathbb{G}(q),\ast)$ the corresponding step two nilpotent Lie group defined as follows. As a manifold we take $\mathbb{G}(q)=\mathfrak{g}(q)$, while the group law is given by
$$\xi_1\ast\xi_2:=\xi_1+\xi_2+\frac{1}{2}[\xi_1,\xi_2] \hspace{2ex} 
\textup{\it for} \hspace{2ex}  \xi_1,\xi_2\in\mathbb{G}(q).$$
\begin{defn}
Given a smooth measure $\mu$ on $M$, its {\it nilpotentization} at $q$  is a measure $\widetilde{\mu}_q$ on $\mathbb{G}(q)$ defined in the chart $\psi$  by
$$\widetilde{\mu}_q:=\lim_{\epsilon\to 0}\frac{1}{\epsilon^{Q}}\delta_\epsilon^\ast\mu.$$

Here the convergence is understood in the weak-$*$-topology of $C_c(M)^\prime$ and $Q$ denotes the Hausdorff dimension of the regular SR manifold $M$ (for more details see \cite{Verd}). Due to the regularity assumption of the SR manifold $M$, the measure $\widetilde{\mu}_q$ is in fact a left-invariant measure on $\mathbb{G}(q)$. Since $\mathbb{G}(q)$ is nilpotent and hence unimodular, the measure $\widetilde{\mu}^q$ is a Haar measure on $\mathbb{G}(q)$ (see \cite{Verd}).
\end{defn}
Now we recall the relation between the first heat invariant $c_0$ and the nilpotentization of the subriemannian manifold $M$. 
As was mentioned in Section \ref{subgeo}, the heat kernel $p(t,\cdot,\cdot)$  has an asymptotic expansion on the diagonal as $t\downarrow  0$ of the form 
$$p(t,q,q)=\frac{1}{t^{Q(q)/2}}\left(c_0(q)+c_1(q)t+\cdots+c_N(q)t^N+o(t^N)\right),$$
for all $N\in\mathbb{N}$ and $q\in M$. Here the (locally defined) smooth coefficients $c_i(q)$ are called {\it heat invariants} of the SR manifold $M$.
\vspace{1ex}\\
Let $p_{\mathbb{G}(q)}$ denote the heat kernel of the intrinsic sublaplacian 
$$\widetilde{\Delta}_{sub}:=\sum_{i=1}^{m}\left(X_i^{(-1)}\right)^2$$
on $\mathbb{G}(q)$ with respect to the Haar measure $\widetilde{\mu}_q$. 
 According to the results in \cite{Verd} the first heat invariant $c_0$ is given by 
\begin{equation}\label{firstheat}
c_0(q)=p_{\mathbb{G}(q)}(1,0,0).
\end{equation}
The interpretation of the second heat invariant $c_1(q)$ through the nilpotentization of the structure at $q$ is given as follows: The sublaplacian $\Delta_{sub}$ on $M$ can be expressed as
$$\Delta_{sub}=-\sum_{i=1}^{m}X_i^2+Y,$$
where $Y:=-\sum_{i=1}^{m}\text{div}(X_i)_{\mathcal{P}}X_i$. Now, we define the dilated sublaplacian $\Delta_{sub}^\epsilon$ by
$$\Delta_{sub}^\epsilon=-\sum_{i=1}^{m}(X_i^\epsilon)^2+\epsilon Y^\epsilon,$$
where $X^\epsilon$, for $X$ a horizontal vector field, is defined in the privileged coordinates by
$$X^\epsilon:=\epsilon \delta_\epsilon^\ast(X).$$
We consider the anisotropic expansion of the horizontal vector fields $X_i^\epsilon$ and $Y^\epsilon$ as $\epsilon\to 0$:
\begin{align*}
& X_i^\epsilon=X_i^{(-1)}+\epsilon X_i^{(0)}+\epsilon^2 X_i^{(1)}+\cdots\\
& Y^\epsilon=Y^{(-1)}+\epsilon Y^{(0)}+\epsilon^2 Y^{(1)}+\cdots
\end{align*}
Then the dilated sublaplacian $\Delta_{sub}^\epsilon$ has an expansion as $\epsilon\to 0$ of the form
$$\Delta_{sub}^\epsilon=-\left(\widetilde{\Delta}_{sub}+\epsilon\mathcal{P}_1+\epsilon^2\mathcal{P}_2+\cdots\right)$$
where $\mathcal{P}_1$ and $\mathcal{P}_2$ are second-order differential operators given by
\begin{align*}
&\mathcal{P}_1:=\sum_{i=1}^{m}\left(X_i^{(-1)}X_i^{(0)}+X_i^{(0)}X_i^{(-1)}\right)+Y^{(-1)}\\
& \mathcal{P}_2:=\sum_{i=1}^{m}X_i^{(-1)}X_i^{(1)}+\sum_{i=1}^{m}X_i^{(1)}X_i^{(-1)}+(X_i^{(0)})^2+Y^{(0)}.
\end{align*}
According to the results in \cite{Verd}, the second heat invariant $c_1(q)$ is given as a convolution integral by
\begin{equation}\label{second_heat_invariant}
c_1(q)=\int_{0}^{1}\int_{\mathbb{G}(q)}p_{\mathbb{G}(q)}(1-s,0,\xi)\mathcal{P}_2\left(p_{\mathbb{G}(q)}(s,\xi,0)\right)d\xi ds,
\end{equation}
where $\mathcal{P}_2$ acts on the second variable.\ \\
In case of a general subriemannian manifold (even step two), we do not know if we can express the second-order differential operator $\mathcal{P}_2$ in terms of geometric data of the underlying manifold at $q$, that is, if $c_1(q)$ has a geometric meaning. However, in the case of $3$D contact SR manifolds, such a result was established by D. Barilari in \cite{Bari}. Moreover, on a strictly pseudoconvex CR manifold with a Levi-metric a similar result holds \cite{BGS}. In both cases, a crucial tool for obtaining a geometric interpretation of $c_1$ was the construction of special privileged coordinates \cite{ACG,Lee}.\ \\
In the following, we consider  special coordinates on qc manifolds constructed in \cite{Kunkel}, the so-called {\it qc normal coordinates}. In these coordinates, by considering a special adapted frame at a given point $q$, we are able to express $\mathcal{P}_2$ as a second-order differential operator with polynomial coefficients in the torsion and curvature tensors of the Biquard connection.
\subsection{QC normal coordinates}
We recall the construction of the qc normal coordinates from \cite{Kunkel}.
\begin{thm}[\cite{Kunkel}]
Let $(M,\mathcal{H},\langle\cdot,\cdot\rangle)$ be a quaternionic contact manifold with Biquard connection $\nabla$ and the decomposition $TM=\mathcal{H}\oplus\mathcal{V}$. Let $q\in M$ and $X_q+V_q\in\mathcal{H}_q\oplus\mathcal{V}_q$. Consider the (geodesic) curve $\gamma_{X_q+V_q}$ starting at $q$ and satisfying
$$D_t^2\dot{\gamma}_{X_q+V_q}=0,\hspace{2mm}\dot{\gamma}_{X_q+V_q}(0)=X_q\text{ and }D_t\dot{\gamma}_{X_q+V_q}(0)=V_q.$$
Then there are neighborhoods $0\in U\subset T_qM$ and $q\in U_M\subset M$ such that the function 
\begin{align*}
\psi:U&\longrightarrow U_M\\
X_q+V_q&\longmapsto\gamma_{X_q+V_q}(1)
\end{align*}
is a diffeomorphism. Furthermore, the following scaling by $t$ holds:
$$\psi(tX_q+t^2V_q)=\gamma_{X_q+V_q}(t)$$
whenever either side is defined. Here $D_t$ denotes the covariant derivative along $\gamma_{X_q+V_q}$.
\end{thm}
Now, let $\{V_1,V_2,V_3\}$ be an oriented basis of $\mathcal{V}_q$ and $I_1,I_2,I_3$ be the associated almost complex structures at $q$. Choose an $Sp(n)Sp(1)$-frame $\{X_1,\cdots,X_{4n}\}$ of $\mathcal{H}_q$. Extending these vectors to be parallel along the geodesics (in the sense of the above theorem) starting at $q$, one obtains a smooth local frame of $TM=\mathcal{H}\oplus\mathcal{V}$. Let us consider the dual frame $\{\theta_1,\cdots,\theta_{4n},\eta_1,\eta_2,\eta_3\}$ of the frame $\{X_1,\cdots,X_{4n},V_1,V_2,V_3\}$. Furthermore, extend the almost complex structures by defining
$$I_i X_{4k+1}=X_{4k+i+1}\text{ for }k=0,\cdots,n-1\text{ and }i=1,2,3.$$
Then one obtains the  so-called special frame and co-frame at $q$. Now, the qc normal coordinates ($x_\alpha,z_i$) at $q$ are defined by composing the inverse of the map $\psi$ with the map 
\begin{align*}
\lambda:T_qM&\longrightarrow \mathbb{R}^{4n+3}\\
X&\longmapsto (\theta_1(X),\cdots,\theta_{4n}(X),\eta_1(X),\eta_2(X),\eta_3(X))^t.
\end{align*}
From now on, the torsion and curvature tensors $T,R$ are considered with respect to this special frame, co-frame and contact forms $\eta_1,\eta_2,\eta_3$. As we will see below, the qc normal coordinates at $q$ are privileged coordinates at $q$, when the qc manifold $M$ is considered as a SR manifold.\ \\
The most important fact is that the infinitesimal generator $$P=\sum_{\alpha=1}^{4n}x_\alpha\frac{\partial}{\partial x_\alpha}+2\sum_{i=1}^{3}z_i\frac{\partial}{\partial z_i}$$
of the action
\begin{align*}
\delta_t:\mathbb{R}^{4n+3}&\longrightarrow\mathbb{R}^{4n+3}\\
(x,z)&\longmapsto(tx,t^2z)
\end{align*}
can be expressed (in qc normal coordinates) in terms of the special frame $\{X_\alpha,V_i\}$ in the form
$$P=\sum_{\alpha=1}^{4n}x_\alpha X_\alpha+\sum_{i=1}^{3}z_i V_i.$$
Using this fact,  the  following expression of the homogeneous terms of the special co-frame in qc normal coordinates was obtained in \cite{Kunkel} (we omit the summation signs for repeated indices) :
\begin{prop}\label{Kunkel}
In the qc normal coordinates centered at $q\in M$, the low order homogeneous  terms of the special co-frame and connection $1$-forms $\omega_{ab}$ are:
\begin{enumerate}
\item $\eta_i^{(2)}=\frac{1}{2}dz_i-I^{i}_{\alpha\beta}x_\alpha dx_\beta,\hspace{2mm}\eta_i^{(3)}=0\text{ and }$
$$\eta_i^{(4)}=\frac{1}{4}\left(z_j \omega_{\overline{j}\hspace{0,5mm}\overline{i}}^{(2)}+T_{\overline{j}\hspace{0,5mm}\overline{k}}^{\overline{i}}(q)z_j\eta_k^{(2)}-2I^{i}_{\alpha\beta}x_\alpha\theta_\beta^{(3)}\right).$$
\item $\theta_\alpha^{(1)}=dx_\alpha,\hspace{2mm}\theta_\alpha^{(2)}=0\text{ and }$
$$\theta_\alpha^{(3)}=\frac{1}{3}\left(x_\beta\omega_{\beta\alpha}^{(2)}-T_{\overline{i}\gamma}^{\alpha}(q)x_\gamma\eta_i^{(2)}+T_{\overline{i}\beta}^{\alpha}(q)z_i\theta_\beta^{(1)}\right).$$
\item $\omega_{ab}^{(1)}=0\text{ and }\omega_{ab}^{(2)}=\frac{1}{2}R_{\alpha\beta a}^{b}(q)x_\alpha\theta_\beta^{(1)}$.
\end{enumerate}
Here the connection 1-forms $\omega_{ab}$ are defined by
$$
\nabla X_\alpha=\omega_{\alpha\beta}\otimes X_\beta,\hspace{1mm}\nabla V_{\overline{i}}=\omega_{\overline{i}\hspace{0,5mm}\overline{j}}\otimes V_{\overline{j}}\text{ and }\omega_{\alpha \overline{i}}=\omega_{\overline{i}\alpha}=0.
$$
\end{prop}

Let us consider the anisotropic expansion of the special frame $\{X_1,\cdots,X_{4n},\\ V_1,V_2,V_3\}$ in the qc normal coordinates $(x_\alpha,z_i)$ near $0$:
\begin{align*}
X_\alpha&=X_\alpha^{(-1)}+X_\alpha^{(0)}+X_\alpha^{(1)}+\cdots\\
V_i&=V_i^{(-2)}+V_i^{(-1)}+V_i^{(0)}+\cdots.
\end{align*}
It was shown in \cite{Kunkel} that the homogeneous terms of the lowest order are given by:
$$X_\alpha^{(-1)}=\frac{\partial}{\partial x_\alpha}+2\sum_{\beta,i}I_{\beta\alpha}^{i}x_\beta\frac{\partial}{\partial z_i}\hspace{2mm}\text{ for }\alpha=1,\cdots,4n$$
and $$V_i^{(-2)}=2\frac{\partial}{\partial z_i}\hspace{2mm}\text{ for }i=1,2,3.$$
Using Proposition \ref{Kunkel}, we compute now the  low order homogeneous terms in the expansion of the special frame.\ \\
 Note that the left-invariant vector fields $$\{X_1^{(-1)},\cdots,X_{4n}^{(-1)},V_1^{(-2)},V_2^{(-2)},V_3^{(-2)}\}$$ on $\mathbb{R}^{4n+3}$ are linearly independent and therefore, every vector field on $\mathbb{R}^{4n+3}$ can be expressed as linear combination of these vector fields.\ \\
In the following we set:
$$\widetilde{X}_\alpha:=X_\alpha^{(-1)}\text{ and }\widetilde{V}_i:=V_i^{(-2)}$$
for $\alpha=1,\cdots,4n$ and $i=1,2,3.$\ \\
The following properties about the order of tensor fields can be proved using (\ref{order}).
\begin{lem}\label{order_pro}
The following hold (whenever it makes sense):
\begin{enumerate}
\item Let $X$ and $Y$ be homogeneous vector fields. Then $[X,Y]$ is homogeneous of order $\text{ord}(X)+\text{ord}(Y)$.
\item let $X$ resp. $\omega$ be a homogeneous vector field resp. $1$-form. Then $\omega(X)$ is homogeneous of order $\text{ord}(\omega)+\text{ord}(X)$.
\item Let $f$ resp. $X$ be a homogeneous function resp.  vector field. Then $fX$ is homogeneous of order $\text{ord}(f)+\text{ord}(X)$.
\end{enumerate}
\end{lem}
Using Lemma \ref{order_pro} we obtain:
 \begin{lem}\label{expansion_vec}
In qc normal coordinates centered at $q$ we have:
\begin{enumerate}
\item For $\alpha=1,\cdots,4n$, it holds: $$X_\alpha^{(0)}=0\text{ and }  X_\alpha^{(1)}=\sum_{\beta}s_\alpha^{\beta(2)} \widetilde{X}_\beta+\sum_{j}r_\alpha^{\overline{j}(3)} \widetilde{V}_j$$  with 
 $$s_\alpha^{\beta(2)}=-\frac{1}{6}\sum_{\gamma,\delta}R_{\gamma\alpha\delta}^{\beta}(q) x_{\gamma}x_{\delta}-\frac{1}{3}\sum_{i}T_{\overline{i}\alpha}^{\beta}(q)z_i$$
 and $$r_\alpha^{\overline{j}(3)}=-\frac{1}{8}\sum_{\gamma,i}R_{\gamma\alpha \overline{i}}^{\overline{j}}(q)x_\gamma z_i+\frac{1}{2}\sum_{\gamma^\prime,\delta^\prime}I_{\gamma^\prime \delta^\prime}^{j}x_{\gamma^\prime}\left(\frac{1}{6}\sum_{\gamma,\delta} R_{\delta\alpha \gamma}^{\delta^\prime}(q)x_\gamma x_\delta+\frac{1}{3}\sum_{k}T_{\overline{k}\alpha}^{\delta^\prime}(q)z_k\right).$$
\item For $i=1,2,3$, it holds:
 $$ V_i^{(-1)}=0\text{ and }   V_i^{(0)}=\sum_{\beta}s_{\overline{i}}^{\beta(1)} \widetilde{X}_\beta+\sum_{j}r_{\overline{i}}^{\overline{j}(2)} \widetilde{V}_j$$
 with 
 $$s_{\overline{i}}^{\beta(1)}=\frac{1}{3}\sum_{\gamma}T_{\overline{i}\gamma}^{\beta}(q)x_{\gamma}$$
 and $$r_{\overline{i}}^{\overline{j}(2)}=-\frac{1}{4}\sum_{k}T_{\overline{k}\hspace{0,5mm}\overline{i}}^{\overline{j}}(q)z_k-\frac{1}{6}\sum_{\gamma,\delta,\delta^\prime}I_{\gamma \delta}^{j}T_{\overline{i}\delta^\prime}^{\delta}(q)x_{\gamma}x_{\delta^\prime}.$$

\end{enumerate}

 \end{lem}
 \begin{proof} 
 \begin{enumerate}
 
\item
 We can write locally near $0$:
 \begin{equation}\label{frame}
 X_\alpha=\sum_{\beta}s_\alpha^\beta \widetilde{X}_\beta+\sum_{j}r_\alpha^{\overline{j}} \widetilde{V}_j
 \end{equation}
 for some smooth functions $s_\alpha^\beta$ and $r_\alpha^{\overline{j}}$ locally defined near $0$. Let us consider the anisotropic expansion of these functions at $0$:
 \begin{align*}
 s_{\alpha}^{\beta}&=s_{\alpha}^{\beta(0)}+s_{\alpha}^{\beta(1)}+s_{\alpha}^{\beta(2)}+\cdots\\
 r_{\alpha}^{\overline{j}}&=r_{\alpha}^{\overline{j}(0)}+r_{\alpha}^{\overline{j}(1)}+r_{\alpha}^{\overline{j}(2)}+\cdots.
 \end{align*}
 Here $s_{\alpha}^{\beta(l)}$ (resp. $r_{\alpha}^{\overline{j}(l)}$) (for $l\geq 0$) denotes the homogeneous term of order $l$ in the expansion of $s_{\alpha}^{\beta}$ (resp. $r_{\alpha}^{\overline{j}}$).
 
  Now, applying $\theta_\gamma$ (resp. $\eta_i$) on both sides of (\ref{frame}) and taking the homogeneous parts of order $l$ (see Lemma \ref{order_pro}), we obtain the following recursive formulas for $s_{\alpha}^{\gamma(l)}$ and $r_{\alpha}^{\overline{i}(l)}$ (for $l\geq 1$): 
  \begin{align*}
  s_{\alpha}^{\gamma(l)}&=-\sum_{m=0}^{l-1}\sum_{\beta}s_\alpha^{\beta(m)}\theta_{\gamma}^{(l-m+1)}(\widetilde{X}_\beta)-\sum_{m=0}^{l}\sum_jr_\alpha^{\overline{j}(m)}\theta_{\gamma}^{(l-m+2)}(\widetilde{V}_j)\\
  r_\alpha^{\overline{i}(l)}&=-\sum_{m=0}^{l-1}\sum_{\beta}s_\alpha^{\beta(m)}\eta_{i}^{(l-m+1)}(\widetilde{X}_\beta)-\sum_{m=0}^{l-1}\sum_jr_\alpha^{\overline{j}(m)}\eta_{i}^{(l-m+2)}(\widetilde{V}_j)
  \end{align*}
   with initial terms $s_\alpha^{\gamma(0)}=\delta_{\alpha\gamma}$ and $r_\alpha^{\overline{i}(0)}=0$.

Combining these recursive formulas with Proposition \ref{Kunkel} we find
$$s_\alpha^{\beta(1)}=0\text{ and }r_\alpha^{\overline{j}(0)}=r_\alpha^{\overline{j}(1)}=r_\alpha^{\overline{j}(2)}=0.$$
Hence, it follows that
  $$s_\alpha^{\beta(2)}=-\theta_{\beta}^{(3)}(\widetilde{X}_\alpha)=-\frac{1}{6}\sum_{\gamma,\delta}R_{\gamma\alpha \delta}^{\beta}(q)x_{\gamma}x_{\delta}-\frac{1}{3}\sum_{i}T_{\overline{i}\alpha}^{\beta}(q)z_i,$$
and
 \begin{align*}
  &r_\alpha^{\overline{j}(3)}=-\eta_j^{(4)}(X_\alpha)\\
  &=-\frac{1}{8}\sum_{\gamma,i}R_{\gamma\alpha \overline{i}}^{\overline{j}}(q)x_\gamma z_i+\frac{1}{2}\sum_{\gamma^\prime,\delta^\prime}I_{\gamma^\prime \delta^\prime}^{j}x_{\gamma^\prime}\left(\frac{1}{6}\sum_{\gamma,\delta} R_{\delta\alpha \gamma}^{\delta^\prime}(q)x_\gamma x_\delta+\frac{1}{3}\sum_{k}T_{\overline{k}\alpha}^{\delta^\prime}(q)z_k\right).
 \end{align*}
 \item The proof is similar to $(1)$ and is left to the reader.
 \end{enumerate}
 \end{proof}
 \begin{rem}
 In qc normal coordinates at $q$, the homogeneous terms of order $0$ of the special frame vanishes, i.e.
 \begin{equation}\label{vanishing}
  X_\alpha^{(0)}=0\text{ for all }\alpha=1,\cdots,4n.
 \end{equation}
 The same holds for the privileged coordinates considered in \cite{Bari} for $3D$ contact manifolds and in \cite{Lee} for CR manifolds. In comparison, in \cite{Wangwu} a system of privileged coordinates for Riemannian contact manifolds was constructed using similar techniques as in \cite{Lee}. But therein, the obstruction for the considered special frame to have the property (\ref{vanishing}) is the obstruction for the Riemannian manifold to be a CR manifold. It is natural to ask in general, if there is some relation between the construction of privileged coordinates for which the property (\ref{vanishing}) holds (if it is possible) and the geometry of the manifold.
 
 \end{rem}
\subsection{First heat invariant}In the following, we assume that the qc manifold $M$ is complete.
For every $q\in M$, the tangent group $\mathbb{G}(q)$ of the subriemannian manifold $M$ at $q$ is isomorphic to the unique connected, simply connected, step two nilpotent Lie group associated to the Lie algebra generated by the vector fields
$$\{\widetilde{X}_1,\cdots,\widetilde{X}_{4n}\}.$$
Using global exponential coordinates, the group law on $\mathbb{G}(q)\simeq \mathbb{R}^{4n+3}$ is given by
$$(x,z)\ast(x^\prime,z^\prime)=(x^{\prime\prime},z^{\prime\prime}),$$
for $(x,z),(x^\prime,z^\prime)\in\mathbb{R}^{4n+3}$ with
$$x^{\prime\prime}_\alpha=x_\alpha+x^\prime_\alpha\text{ and }z_i^{\prime\prime}=z_i+z_i^\prime+2\sum_{\alpha,\beta=1}^{4n}I^{i}_{\alpha\beta}x_\alpha x_{\beta}^{\prime}.$$
Furthermore, by definition and according to Proposition \ref{Kunkel}, the nilpotentization of the Popp measure $\mathcal{P}$ at $q$ is the Haar measure $\widetilde{\mathcal{P}}_q$ on $\mathbb{G}(q)$ given by
\begin{align*}
\widetilde{\mathcal{P}}_q&=\frac{1}{(16n)^{3/2}}\theta_1^{(1)}\wedge\cdots\wedge\theta_{4n}^{(1)}\wedge\eta_1^{(2)}\wedge\eta_2^{(2)}\wedge\eta_3^{(2)}\\
&=\frac{1}{8(16n)^{3/2}}dx_1\wedge\cdots\wedge dx_{4n}\wedge dz_1\wedge dz_2\wedge dz_3.
\end{align*}
In order to compute the first heat invariant $c_0(q)$ we need to derive the heat kernel $p_{\mathbb{G}(q)}$ of the sublaplacian
$$\widetilde{\Delta}_{sub}:=-\sum_{\alpha=1}^{4n}\widetilde{X}_\alpha^2$$
on $\mathbb{G}(q)$ with respect to the Haar measure $\widetilde{\mathcal{P}}_q$. Explicitly, this is obtained by the {\it Beals-Gaveau-Greiner formula} for the sublaplacian  on  general  step two nilpotent Lie groups in \cite{BGG,CCFI}, which we recall next.
For  $h,h^\prime\in\mathbb{G}(q)$ it holds:
\begin{equation}\label{beals}
p_{\mathbb{G}(q)}(t,h,h^\prime)=\frac{8(16n)^{3/2}}{(4\pi
t)^{2n+3}}\int_{\mathbb{R}^3}e^{-\frac{\varphi(\tau,\,h^{-1}\: * \: h^\prime)}{t}}\,W(\tau)\,d\tau,
\end{equation}
where the  {\it action function} $\varphi=\varphi(\tau,h)\in C^{\infty}(\mathbb{R}^{3}\times
\mathbb{G}(q))$ and the  {\it volume element} $W(\tau)\in C^{\infty}(\mathbb{R}^3)$ are given as follows:
Put $h=(x,z)\in \mathbb{R}^{4n}\times\mathbb{R}^3$, then
\begin{align*}
&\varphi(\tau,h)=\varphi(\tau,x,z)=\sqrt{-1}\langle \tau,\,z \rangle 
+\frac{1}{2}\Big{\langle}\sqrt{-1}\Omega_{\tau}\coth\bigr(\sqrt{-1}\Omega_{\tau}\bigr)
\cdot x,x\Big{\rangle},\\
&W(\tau)=
\left\{\det\frac{\sqrt{-1}\Omega_{\tau}}{\sinh \sqrt{-1}\Omega_{\tau}}\right\}^{1/2},
\end{align*}
where $\langle \tau,\tau^{\prime}\rangle=\sum\limits_{i=1}^3 \tau_i \tau_i^{\prime}$ denotes the Euclidean inner product on $\mathbb{R}^3$ and the matrix $\Omega_\tau$ encodes the structure constants of the Lie algebra and is given by 
$$\Omega_\tau=2\sum_{i=1}^{3}\tau_i(I^{i}_{\alpha\beta})_{\alpha,\beta}\in\mathbb{R}^{4n\times 4n}.$$
Using the relations (\ref{almost}), we see that the eigenvalues of $\sqrt{-1}\Omega_\tau$ are $\pm 2\|\tau\|$. Hence the functions $\varphi(\tau,h)$ and $W(\tau)$ take the forms
\begin{align*}
&\varphi(\tau,h)=\varphi(\tau,x,z)=\sqrt{-1}\langle \tau,\,z \rangle 
+\frac{1}{2}\left(\|2\tau\|\coth{\|2\tau\|}\right)\|x\|^2,\\
&W(\tau)=\left(\frac{\|2\tau\|}{\sinh{\|2\tau\|}}\right)^{2n}.
\end{align*}
Recalling that the first heat invariant is given by
$$c_0(q)=p_{\mathbb{G}(q)}(1,0,0),$$
we have the following formula:
\begin{thm}\label{fund1}
 The first heat invariant $c_0(q)$ of the complete qc manifold $M$ is independent of the point $q\in M$ and is given by
$$c_0(q)=\frac{(16n)^{3/2}}{(4\pi)^{2n+3}}\int_{\mathbb{R}^3}\left(\frac{\|\tau\|}{\sinh{\|\tau\|}}\right)^{2n}d\tau.$$
\end{thm}
\begin{rem}
The expression for $c_0$ obtained in \cite{Baud_Wang} for the case where $M=\mathbb{S}^{4n+3}$ with the standard qc structure differs from our formula by the factor $(16n)^{3/2}$. This is due to the fact that our sublaplacian is defined with respect to the Popp measure, which differs from the Riemannian measure on $\mathbb{S}^{4n+3}$ (with respect to the standard metric) by the factor $\frac{1}{(16n)^{3/2}}$ (see Lemma \ref{measure}).

\end{rem}

\subsection{Second heat invariant}

In general, computing the remaining heat invariants $c_1,c_2,\cdots$ with the help of  the nilpotentization is rather complicated. However, using the qc normal coordinates on a qc manifold, it is possible to give a geometric meaning to the second heat invariant $c_1(q)$ and this is the main goal of the present paper.\ \\
For this, by (\ref{second_heat_invariant}), Lemma \ref{sublap} and Lemma \ref{expansion_vec}, additionally we need the low order homogeneous terms in the expansion of the horizontal vector field $\epsilon Y^\epsilon=\epsilon^2 \delta_\epsilon^\ast(Y)$, where
$$Y:=\sum_{\alpha=1}^{4n}\left(\sum_{a=1}^{4n+3}c_{a\alpha}^{a}\right)X_\alpha.$$
A straightforward calculation shows that 
$$\epsilon Y^\epsilon=\sum_{\alpha=1}^{4n}\left(\sum_{a=1}^{4n+3}c_{a\alpha}^{a}(\epsilon)\right)X_\alpha^\epsilon,$$
where $c_{a\alpha}^{a}(\epsilon)$ is defined by
$$c_{a\alpha}^{a}(\epsilon):=\theta_a^{\epsilon}([X_a^\epsilon,X_\alpha^\epsilon])$$
with $\theta_a^\epsilon:=\frac{1}{\epsilon}\delta_\epsilon^\ast(\theta_a)$.

\begin{lem}\label{first_order_term}
Let $\alpha=1,\cdots,4n$. As $\epsilon\to 0$, it holds:
$$\sum_{a=1}^{4n+3}c_{a\alpha}^{a}(\epsilon)=\epsilon^2\left(\sum_{\beta=1}^{4n}\widetilde{X}_\beta(s_\alpha^{\beta(2)})-\widetilde{X}_\alpha(s_\beta^{\beta(2)})+\sum_{i=1}^{3}\widetilde{V}_i(r_\alpha^{\overline{i}(3)})-\widetilde{X}_\alpha(r_{\overline{i}}^{\overline{i}(2)})\right)+O(\epsilon^3).$$
\end{lem}
\begin{proof}
By Proposition \ref{Kunkel} and Lemma \ref{expansion_vec}, as $\epsilon\to 0$ it holds:
$$X_\gamma^\epsilon=\widetilde{X}_\gamma+\epsilon^2 X_\gamma^{(1)}+O(\epsilon^3)$$
and $$\theta_\gamma^\epsilon=\theta_\gamma^{(1)}+\epsilon^2\theta_\gamma^{(3)}+O(\epsilon^3)$$ for $\gamma=1,\cdots,4n$.
Hence, we can write
\begin{align*}
c_{\beta\alpha}^{\beta}(\epsilon)&=\theta_\beta^{\epsilon}([X_\beta^\epsilon,X_\alpha^\epsilon])\\
&=\left(\theta_\beta^{(1)}+\epsilon^2\theta_\beta^{(3)}+O(\epsilon^3)\right)\left([\widetilde{X}_\beta,\widetilde{X}_\alpha]+\epsilon^2[\widetilde{X}_\beta,X_\alpha^{(1)}]+\epsilon^2[X_\beta^{(1)},\widetilde{X}_\alpha]+O(\epsilon^4)\right)\\
&=\theta_\beta^{(1)}([\widetilde{X}_\beta,\widetilde{X}_\alpha])+\epsilon^2 \left(\theta_\beta^{(1)}\left([\widetilde{X}_\beta,X_\alpha^{(1)}]+[X_\beta^{(1)},\widetilde{X}_\alpha]\right)+\theta_\beta^{(3)}\left([\widetilde{X}_\beta,\widetilde{X}_\alpha]\right)\right)+O(\epsilon^3).
\end{align*}
Using the fact that $\theta_\beta^{(1)}=dx_\beta$ and that the Lie brackets of two horizontal left-invariant vector fields is a vertical vector field, i.e. 
$$\theta_\beta^{(1)}([\widetilde{X}_\beta,\widetilde{X}_\alpha])=dx_\beta\left(\sum_{i=1}^{3}2I_{\beta\alpha}^{i}\frac{\partial}{\partial z_i}\right)=0,$$
it follows that
\begin{equation}\label{strct}
c_{\beta\alpha}^{\beta}(\epsilon)=\epsilon^2 \left(\theta_\beta^{(1)}\left([\widetilde{X}_\beta,X_\alpha^{(1)}]+[X_\beta^{(1)},\widetilde{X}_\alpha]\right)+\theta_\beta^{(3)}\left([\widetilde{X}_\beta,\widetilde{X}_\alpha]\right)\right)+O(\epsilon^3).
\end{equation}

We recall that by Lemma \ref{expansion_vec}, it holds for $\delta=1,\cdots,4n$:
$$X_\delta^{(1)}=\sum_{\gamma}s_\delta^{\gamma(2)} \widetilde{X}_\gamma+\sum_{j}r_\delta^{\overline{j}(3)} \widetilde{V}_j.$$
Inserting the last identity into Equation (\ref{strct}) and using the fact that $\theta_\beta^{(1)}(\widetilde{V}_i)=0$ for all $i$ and $\beta$, we obtain
\begin{equation}\label{strc1}
c_{\beta\alpha}^{\beta}(\epsilon)=\epsilon^2\left(\widetilde{X}_\beta(s_\alpha^{\beta(2)})-\widetilde{X}_\alpha(s_\beta^{\beta(2)})+\theta_\beta^{(3)}\left([\widetilde{X}_\beta,\widetilde{X}_\alpha]\right)\right)+O(\epsilon^3).
\end{equation}
In the same way, we can prove that for $i=1,2,3$:
\begin{equation}\label{strc2}
c_{\overline{i}\alpha}^{\overline{i}}(\epsilon)=\epsilon^2\left(\widetilde{V}_i (r_\alpha^{\overline{i}(3)})-\widetilde{X}_\alpha (r_{\overline{i}}^{\overline{i}(2)})+\sum_{\beta=1}^{4n}s_{\overline{i}}^{\beta(1)}\eta_i^{(2)}\left([\widetilde{X}_\beta,\widetilde{X}_\alpha]\right)\right)+O(\epsilon^3).
\end{equation}
Again by using Proposition \ref{Kunkel} and Lemma \ref{expansion_vec}, a straightforward calculation shows that
$$\sum_{\beta=1}^{4n}\theta_\beta^{(3)}\left([\widetilde{X}_\beta,\widetilde{X}_\alpha]\right)=-\sum_{\beta=1}^{4n}\sum_{i=1}^{3}s_{\overline{i}}^{\beta(1)}\eta_i^{(2)}\left([\widetilde{X}_\beta,\widetilde{X}_\alpha]\right).$$
Now the assertion follows by adding Equations (\ref{strc1}) and (\ref{strc2}).
\end{proof}
Using Lemma \ref{first_order_term} and Lemma \ref{sublap} we can write
$$\Delta_{sub}^{\epsilon}=\widetilde{\Delta}_{sub}+\epsilon^2\mathcal{P}_2+\epsilon^3\mathcal{R}^\epsilon_q,$$
where $\mathcal{P}_2$ and $\mathcal{R}^\epsilon_q$ are second-order differential operators with (we omit the summation signs for repeated indices)
\begin{equation}\label{second order}
\mathcal{P}_2=\widetilde{X}_\alpha X_\alpha^{(1)}+X_\alpha^{(1)}\widetilde{X}_\alpha+\left(\widetilde{X}_\beta(s_\alpha^{\beta(2)})-\widetilde{X}_\alpha(s_\beta^{\beta(2)})+\widetilde{V}_i(s_\alpha^{\overline{i}(3)})-\widetilde{X}_\alpha(s_{\overline{i}}^{\overline{i}(2)})\right)\widetilde{X}_\alpha.
\end{equation}
We recall that the second heat invariant $c_1(q)$ can be expressed in the from
\begin{equation}\label{second_invariant}
c_1(q)=\int_{0}^{1}\int_{\mathbb{R}^{4n+3}}p_{\mathbb{G}(q)}(1-s,0,(x,z))\mathcal{P}_2\left( p_{\mathbb{G}(q)}(s,(x,z),0)\right) dxdzds.
\end{equation}
Note also that the heat kernel
$$p_{\mathbb{G}(q)}(t,0,(x,z))=\frac{(16n)^{3/2}}{(4\pi t)^{2n+3}}\int_{\mathbb{R}^3}e^{-(\sqrt{-1}\langle \tau,z\rangle+\frac{1}{2}\|\tau\|\coth{(\|\tau\|)}\|x\|^2)/t}\left(\frac{\|\tau\|}{\sinh{\|\tau\|}}\right)^{2n}d\tau$$
is invariant under the action of the orthogonal group ${\bf O}(4n)$ (resp. ${\bf O}(3)$) on the variable $x$ (resp. $z$).\ \\
According to (\ref{second order}), the second-order differential operator $\mathcal{P}_2$ can be completely expressed through the vector fields
$$\widetilde{X}_\alpha,\hspace{1mm} X_\alpha^{(1)},\hspace{1mm} \widetilde{V}_i,\hspace{1mm} V_{i}^{(0)}.$$
Furthermore, by Lemma \ref{expansion_vec} these vector fields can be written in terms of the torsion and curvature tensors of the Biquard connection at $q$ and the left-invriant vector fields $\widetilde{X}_\alpha,\widetilde{V}_i$. Hence, it follows from (\ref{second_invariant}) that $c_1(q)$ can be expressed as a polynomial in the torsion and curvature tensors at $q$. In fact, using the ${\bf O}(4n)\times {\bf O}(3)$-invariance of $p_{\mathbb{G}(q)}$, more can be proved:
\begin{thm}\label{fund}
Let $(M,\mathcal{H},\langle\cdot,\cdot\rangle)$ be a complete quaternionic contact manifold and let $q\in M$. Then it holds:
$$p(t,q,q)=\frac{1}{t^{2n+3}}\left(c_0+C_n\kappa(q)t+o(t)\right)\text{ as }t\to 0.$$
Here $\kappa(q)$ denotes the qc scalar curvature of the Biquard connection (see Definition \ref{defi}) and $C_n$ is a universal constant depending only on $n$ and independent of the qc manifold $M$.
\end{thm}
\begin{proof}
Note that our structure has two symmetries, the first one being the $Sp(n)Sp(1)$-action on the $Sp(n)Sp(1)$-frame $\{X_\alpha\}$ and the second one the $SO(3)$-action on the vertical frame $\{V_i\}$. That is, since $c_1(q)$ depends only on the point $q$ and not on the choice of the special frame $\{X_\alpha,V_i\}$, the expression of $c_1(q)$ which involves the torsion and curvature tensors at $q$ must be invariant under the action of these groups. Therefore, we could use these symmetries to find the explicit dependence of $c_1(q)$ on the torsion and curvature tensors at $q$ (see \cite{BGS} for the case of CR manifolds). However, we rather use the symmetries of the heat kernel $p_{\mathbb{G}(q)}$ on the quaternionic Heisenberg group and standard identities of torsion and curvature tensors like the ones in Proposition \ref{identities} to obtain our result.  In the following we omit the summation signs for repeated indices to simplify the notations. We set
 $$\mathcal{P}_2=\mathcal{P}_{21}+\mathcal{P}_{22},$$
 where
 \begin{align*}
 \mathcal{P}_{21}:&=\widetilde{X}_\alpha X_\alpha^{(1)}+X_\alpha^{(1)}\widetilde{X}_\alpha\\
 &=\widetilde{X}_\alpha(s_\alpha^{\beta(2)})\widetilde{X}_\beta+\widetilde{X}_\alpha(r_\alpha^{\overline{j}(3)})\widetilde{V}_j+2r_\alpha^{\overline{j}(3)}\widetilde{X}_\alpha\widetilde{V}_j+s_\alpha^{\beta(2)}\left(\widetilde{X}_\beta\widetilde{X}_\alpha+\widetilde{X}_\alpha\widetilde{X}_\beta\right).
 \end{align*}
 and 
 $$\mathcal{P}_{22}:=\left(\widetilde{X}_\beta(s_\alpha^{\beta(2)})-\widetilde{X}_\alpha(s_\beta^{\beta(2)})+\widetilde{V}_i(r_\alpha^{\overline{i}(3)})-\widetilde{X}_\alpha(r_{\overline{i}}^{\overline{i}(2)})\right)\widetilde{X}_\alpha.$$
 Let us consider for example the last summation in the expression of $\mathcal{P}_{21}$:
 $$s_{\alpha}^{\beta(2)}\widetilde{X}_\alpha\widetilde{X}_\beta.$$
 Using  the ${\bf O}(4n)\times{\bf O}(3)$-invariance and the exponential decay of the heat kernel $p_{\mathbb{G}(q)}$ and parity arguments, we obtain the following formulas:

\begin{enumerate}
\item For $\alpha,\beta=1,\cdots,4n$ and $i=1,2,3$:
$$\int_{\mathbb{R}^{4n+3}}p_{\mathbb{G}(q)}(1-s,0,(x,z))x_\alpha x_\beta\frac{\partial}{\partial z_i}\left( p_{\mathbb{G}(q)}(s,(x,z),0)\right) dxdz=0,$$
\item  For $\alpha,\beta,\gamma,\delta=1,\cdots,4n$:
$$\int_{\mathbb{R}^{4n+3}}p_{\mathbb{G}(q)}(1-s,0,(x,z))x_\alpha x_\beta\frac{\partial}{\partial x_\gamma}\frac{\partial}{\partial x_\delta}\left( p_{\mathbb{G}(q)}(s,(x,z),0)\right) dxdz=0,$$
unless one of the following cases holds:
$$\alpha=\beta\text{ and }\gamma=\delta,\hspace{2mm}\alpha=\gamma\text{ and }\beta=\delta\text{ or }\alpha=\delta\text{ and }\beta=\gamma.$$
\item For $i=1,2,3$:
$$\int_{\mathbb{R}^{4n+3}}p_{\mathbb{G}(q)}(1-s,0,(x,z))\frac{\partial}{\partial z_i}\left( p_{\mathbb{G}(q)}(s,(x,z),0)\right) dxdz=0.$$
\item For $\alpha,\beta,\gamma,\delta=1,\cdots,4n$ and  $i,j=1,2,3$:
$$\int_{\mathbb{R}^{4n+3}}p_{\mathbb{G}(q)}(1-s,0,(x,z))x_\alpha x_\beta x_\gamma x_\delta\frac{\partial}{\partial z_i}\frac{\partial}{\partial z_j}\left( p_{\mathbb{G}(q)}(s,(x,z),0)\right) dxdz=0,$$
unless $i=j$ and one of the following cases hold:
$$\alpha=\beta\text{ and }\gamma=\delta,\hspace{2mm}\alpha=\gamma\text{ and }\beta=\delta\text{ or }\alpha=\delta\text{ and }\beta=\gamma.$$
In each case, the above integral does not depend on $i=j$ and $\alpha,\beta,\gamma,\delta$.
\end{enumerate}
Now, using the formulas cited above, the remaining terms in the expression of 
\begin{equation}\label{example}
\int_{0}^{1}\int_{\mathbb{R}^{4n+3}}p_{\mathbb{G}(q)}(1-s,0,(x,z))s_{\alpha}^{\beta(2)}\widetilde{X}_\alpha\widetilde{X}_\beta\left( p_{\mathbb{G}(q)}(s,(x,z),0)\right) dxdzds
\end{equation}

are linear combination of 
\begin{equation}\label{remaining}
R_{\alpha\beta\beta}^{\alpha}(q),I_{\alpha\beta}^{i}T_{\overline{j}\alpha}^{\beta}(q),I_{\alpha\beta}^{i}I_{\gamma\delta}^{i}R_{\alpha\beta\gamma}^{\delta}(q)\text{ and }I_{\alpha\beta}^{i}I_{\gamma\delta}^{i}R_{\gamma\alpha\beta}^{\delta}(q),
\end{equation}
where no summation over $i$ is implied in the third and fourth terms.\ \\
The first term of (\ref{remaining}) is the qc scalar curvature at $q$: $R_{\alpha\beta\beta}^{\alpha}(q)=\kappa(q)$. Furthermore, using Proposition \ref{identities} and the fact that $I_iX_\alpha=I_{\beta\alpha}^{i}X_\beta$, the remaining terms of (\ref{remaining}) can be expressed  in the form
\begin{align*}
I_{\alpha\beta}^{i}T_{\overline{j}\alpha}^{\beta}(q)&=-g(T(V_j,X_\alpha),I_i X_\alpha)=0\\
I_{\alpha\beta}^{i}I_{\gamma\delta}^{i}R_{\gamma\alpha\beta}^{\delta}(q)&=-g(R(X_\beta,I_iX_\alpha)X_\alpha,I_iX_\beta)=\frac{n\kappa(q)}{n+2},\text{ for all }i\\
I_{\alpha\beta}^{i}I_{\gamma\delta}^{i}R_{\alpha\beta\gamma}^{\delta}(q)&=-g(R(X_\alpha,I_iX_\alpha)I_iX_\beta,X_\beta)=-\frac{2n\kappa(q)}{n+2},\text{ for all }i.
\end{align*}
It follows that each non-trivial term of (\ref{remaining}) equals the qc scalar curvature at $q$ up to a constant multiple. Hence (\ref{example}) depends linearly on the curvature $\kappa(q)$ at $q$.\ \\ Similar arguments show that
\begin{align*}
\mathcal{I}_1(q):&=\int_{0}^{1}\int_{\mathbb{R}^{4n+3}}p_{\mathbb{G}(q)}(1-s,0,(x,z))\mathcal{P}_{21}\left( p_{\mathbb{G}(q)}(s,(x,z),0)\right) dxdzds\\
&=C_1(n)\kappa(q)
\end{align*}
where $C_1(n)$ is a constant depending only on $n$.

Again, similar arguments show that 
$$\mathcal{I}_2(q):=\int_{0}^{1}\int_{\mathbb{R}^{4n+3}}p_{\mathbb{G}(q)}(1-s,0,(x,z))\mathcal{P}_{22}\left( p_{\mathbb{G}(q)}(s,(x,z),0)\right) dxdzds$$
is a linear combination of the following terms:
$$R_{\alpha\beta\beta}^{\alpha}(q),I_{\alpha\beta}^{i}T_{\overline{j}\alpha}^{\beta}(q)\text{ and }T_{\overline{j}\hspace{0,5mm}\overline{i}}^{\overline{i}}(q),$$
where no summation over $i$ is implied in the third term.\ \\
Using Proposition \ref{identities} shows that the second term $I_{\alpha\beta}^{i}T_{\overline{j}\alpha}^{\beta}(q)$ and the last term $T_{\overline{j}\hspace{0,5mm}\overline{i}}^{\overline{i}}(q)$ vanish. Therefore, we conclude that the second heat invariant $c_1(q)$ depends linearly on the curvature at $q$ with coefficient depending only on $n$.
\end{proof}

For $M=\mathbb{S}^{4n+3}$ with the standard qc structure, the second coefficient $c_1$ was computed in \cite{Baud_Wang}:
$$c_1(q)=\frac{1}{(4\pi)^{2n+2}}\int_{0}^{\infty}\frac{y^{2n+2}}{(\sinh y)^{2n}}\left((2n+1)^2-\frac{2n(2n+1)(\sinh y-y\cosh y)}{y^2\sinh y}\right)dy,$$
for $q\in\mathbb{S}^{4n+3}$. Furthermore, the qc curvature $\kappa_{\mathbb{S}^{4n+3}}$ of the sphere $\mathbb{S}^{4n+3}$ is constant with value $16n(n+2)$ (see  \cite{IMV}). Hence we deduce the value of the universal constant $C_n$ in Theorem \ref{fund}: 
$$\mbox{\normalsize $C_n=\frac{1}{16(n^2+2n)(4\pi)^{2n+2}}\int_{0}^{\infty}\frac{y^{2n+2}}{(\sinh y)^{2n}}
 \left((2n+1)^2-\frac{2n(2n+1)(\sinh y-y\cosh y)}{y^2\sinh y}\right)dy.$}$$
 
In the following we give an application of Theorem \ref{fund} for the case of qc-Einstein compact manifolds. We recall that a qc-Einstein manifold $M$ is a qc manifold such that the torsion $T_V$ vanishes identically on $\mathcal{H}$. In this case, it was proven in \cite{IMV} that the qc scalar curvature $\kappa$ of such a manifold is constant. Note that if the qc manifold $M$ is compact, then the sublaplacian $\Delta_{sub}$ (which is subelliptic due to the bracket generating property of the distribution $\mathcal{H}$) has compact resolvent and thus has discrete spectrum $0=\lambda_1\leq \lambda_2 \leq \cdots \leq \lambda_k \cdots \to\infty$ only consisting of eigenvalues with finite multiplicities. Furthermore, $e^{-t\Delta_{sub}}$ is of trace class for every $t>0$.\ \\
Now, let $M$ and $M^\prime$ be two qc-Einstein compact manifolds with qc scalar curvatures $\kappa$ and $\kappa^{\prime}$, respectively. Assume that $M$ and $M^{\prime}$ are isospectral with respect to the intrinsic sublaplacians, i.e. the  associated intrinsic sublaplacians $\Delta_{sub}$ and $\Delta_{sub}^\prime$  have the same spectrum with the same multiplicities of eigenvalues. Hence we can write
$$\text{tr}(e^{-t\Delta_{sub}})=\text{tr}(e^{-t\Delta_{sub}^\prime}).$$
By Theorems \ref{fund1} and \ref{fund} we obtain
$$\dim(M)=\dim(M^\prime),\hspace{1mm}\mathcal{P}(M)=\mathcal{P}^\prime(M^\prime)\text{ and }\kappa=\kappa^\prime,$$
i.e. the dimension, the Popp volume and the qc scalar curvature of a qc-Einstein compact manifold are spectral invariants.
\begin{cor}
Let  $M$ and $M^\prime$ be two qc-Einstein compact manifolds, which are isospectral with respect to the intrinsic sublaplacians. Then they have the same dimension, Popp volume and qc scalar curvature.
\end{cor}

%%%%%%%%%%%%%%%%%%%%%%%%%%%%%%%%%%%%%%%%%%%%%%%%%%%%%%%%%%%%%%%%%%%%%%%%%

%%%%%%%%%%%%%%%%%%%%%%%%%%%%%%%%%%%%%%%%%%%%%%%%%%%%%%%%%%%%%%%%%%%%%%%%%%%

%%%%%%%%%%%%%%%%%%%%%%%%%%%%%%%%%%%%%%%%%%%%%%%%%%%%%%%%%%%%%%%%%%%%%%%%%%%%%%%%%%%%%%%%%%%%%%%
 
\end{document}